\documentclass[10pt,leqno]{amsart}
\usepackage{amsmath}
\usepackage{amsfonts}
\usepackage{amssymb}
\usepackage{mathtools}
\usepackage{graphicx}
\usepackage{color}
\usepackage{hyperref}
\allowdisplaybreaks[4]

\usepackage{xcolor}
\newtheorem{theorem}{Theorem}[section]
\newtheorem*{theorem*}{Theorem}
\newtheorem{lemma}[theorem]{Lemma}

\newtheorem{proposition}[theorem]{Proposition}

\theoremstyle{definition}
\newtheorem{definition}[theorem]{Definition}

\newcommand{\R}{\mathbb{R}}

\def \b {\beta}

\def\Ric{\text{Ric}}
\def\lf{\left}
\def\ri{\right}

\def\a{\alpha}
\def\l{\lambda}

\def\R{\mathbb{R}}

\def\S{{\operatorname{Scal}}}

\def\vp{\varphi}

\def\Ric{\operatorname{Ric}}

\def\tr{\operatorname{tr}}

\numberwithin{equation}{section}
\makeatletter
\newcommand*\owedge{\mathpalette\@owedge\relax}
\newcommand*\@owedge[1]{%
  \mathbin{%
    \ooalign{%
      $#1\m@th\bigcirc$\cr
      \hidewidth$#1\m@th\wedge$\hidewidth\cr
    }%
  }%
}
\makeatother

\begin{document}

\title[Secondary Curvature Operator on Einstein Manifolds]{
Einstein Manifolds Under  Cone Conditions for the Curvature Operator of the Second Kind}

\author[Cheng]{Haiqing Cheng}
\address{School of Mathematical Sciences, Soochow University, Suzhou, 215006, China}
\email{chq4523@163.com}

\author[Wang]{Kui Wang}\thanks{}
\address{School of Mathematical Sciences, Soochow University, Suzhou, 215006, China}
\email{kuiwang@suda.edu.cn}



\begin{abstract}
It is established in \cite{CGT23, Li22JGA, NPW22} that any closed Einstein manifold with two-nonnegative curvature operator of the second kind is either flat or a round sphere. In this paper, we  refine this result by relaxing  the 
curvature condition  to   a cone condition (strictly weaker than two nonnegativity) proposed by Li \cite{Li24}. Precisely,  we prove that any closed Einstein manifold of dimension $n=4$ or $n=5$ or $n\ge 8$, if the curvature operator of the second kind $\mathring{R}$ satisfies 
\begin{align*}
    (\lambda_1+\lambda_2)/2 \ge -\theta(n) \bar \lambda, 
\end{align*}
then the manifold is either flat or a round sphere. Here, $\lambda_1\le \lambda_2\le \cdots\le \lambda_{(n-1)(n+2)/2}$ are the eigenvalues of $\mathring{R}$,  $ \bar \lambda $ is their average, and $\theta(n)$ is a positive constant defined as in \eqref{1.2}.  
\end{abstract}
\maketitle

\section{Introduction and main results}
The classification  for manifolds under various curvature conditions is one of the central topics in geometric analysis. 
There are many remarkable results in this area obtained by various authors (cf. \cite{BW08, BS09, BS11, Ham82, NW10, NW07, PW21}). 
For instance: Meyer \cite{DM71} used the Bochner technique to prove that a closed Riemannian manifold with positive curvature operator of the first kind must be a real homology sphere;  Hamilton \cite{Ham82} used Ricci flow to prove that manifolds of dimension three with positive Ricci curvature are diffeomorphic to spherical space forms; Brendle and Schoen \cite{BS09} proved a differentiable sphere theorem for quarter-pinched metrics.

In 1986, Nishikawa \cite{Nishikawa86} conjectured that a closed Riemannian manifold with positive (respectively, nonnegative) curvature operator of the second kind, as defined in Section \ref{sec 2.2} below, is diffeomorphic to a spherical space form (respectively, a Riemannian locally symmetric space).
Recently, this conjecture has been confirmed by  Cao-Gursky-Tran \cite{CGT23} under the  assumption of two-positive curvature operator of the second kind, by Li \cite{Li21} and Nienhaus-Petersen-Wink \cite{NPW22} under a weaker assumption of three-nonnegativity, yielding the stronger conclusion that the manifold is either flat or  diffeomorphic to a spherical space form, see also \cite[Theorem 1.1]{Li24}.
Following the resolution of Nishikawa's conjecture, several classification results  were proved under various conditions of curvature operator of the second kind, cf. \cite{DF24, DFY24, Li22JGA, Li22Kahler, Li22PAMS,   Li24, Li22product,  NPW22}.

For Einstein manifolds, there are many excellent results concerning the curvature operator of the second kind. Kashiwada \cite{Kashiwada93} showed that closed Einstein manifolds with positive curvature operator of the second kind are spheres. 
Cao-Gursky-Tran \cite{CGT23} proved that for Einstein manifolds, four-positive (respectively, four-nonnegative) curvature operator of the second kind implies constant sectional curvature (respectively, local symmetry). 
Li \cite{Li22JGA} generalized this to $4\frac12$-positive (respectively, $4\frac12$-nonnegative) curvature operator of the second kind.
Later, Nienhaus-Petersen-Wink \cite{NPW22} 
proved that any $n$-dimensional compact Einstein manifold with $M$-nonnegative ($ M< \frac{3n(n+2)}{2(n+4)}$) curvature operator of the second kind is either flat or a rational homology sphere, by developing a Bochner formula for the curvature of the second kind. Recently,  Dai-Fu  \cite{DF24} used  the Bochner formula to give an alternative proof   of that any  closed Einstein manifold of dimension $n=4$ or $n=5$ or $n\ge 8$ with  two-nonnegative curvature operator of the second kind is a constant curvature space.
For further information and recent developments regarding this topic, we refer readers to \cite{CW24, FL24, Li22JGA,  Li22Kahler, Li22PAMS, Li22product, NPWW22}.

Throughout this paper, we denote  by $\mathring{R}$ the curvature operator of the second kind on a Riemannian manifold $M^n$, and  by $N=(n+2)(n-1)/2$  the dimension of the trace-free symmetric two-tensor space $S_0^2(T_p M)$, and by $\l_1\le \l_2\le \cdots\le \l_N$ the eigenvalues of  $\mathring{R}$ and by $\bar \l$ the average of these eigenvalues.

In recent work, Li \cite{Li24} studied  cone conditions of the curvature operator of the second kind, denoted by $\mathcal{C}(\a, \theta)$, which  satisfies 
\begin{align}\label{eq:conecond}
\alpha^{-1} \left( \lambda_1 + \cdots + \lambda_\alpha \right) \geq -\theta \bar{\lambda},
\end{align}
and proved some new  sphere theorems under these cone conditions.
Here $\a\in[1, N)$, $\theta>-1$, and  
\begin{align*}
\lambda_1 + \cdots + \lambda_\alpha := \displaystyle\sum_{i=1}^{\lfloor \alpha \rfloor}
\lambda_i+ (\alpha - \lfloor \alpha \rfloor) \lambda_{\lfloor \alpha \rfloor + 1},
\end{align*}
where $\lfloor \alpha \rfloor:=\max\{ m \in \mathbb{Z}: m \leq \a\}$.
As mentioned above, based on previous results of \cite{CGT23, Li22JGA, NPW22}, we conclude that for any closed Einstein manifold of dimension $n$ satisfying $\mathring{R}\in \mathcal{C}(2, 0)$ (i.e. the curvature operator of the second kind is two-nonnegative) is  either  flat  or a round sphere.
The main purpose of this paper is to study the  Einstein manifolds with the cone condition $\mathring{R}\in \mathcal{C}(2, \theta)$ for $\theta>0$, which is weaker than  two-nonnegative curvature operator of the second kind condition. Throughout this paper, we set
\begin{align}\label{1.2}
    \theta(n)=\begin{cases}
        1/2, &\quad n=4,\\
        2/3, &\quad n=5,\\
      \frac{2N-9n+6}{N+6n-3}, &\quad n\ge 8.
    \end{cases}
\end{align}
Our main result is the following  theorem:
\begin{theorem}\label{thm1}
Let $(M^n,g)$ be an $n$-dimensional closed Einstein manifold. For $n=4$ or $n=5$, or $n\ge 8$, if  the curvature operator of the second kind satisfies
\begin{align}\label{1.1}
\l_1+\l_2\ge -2\theta(n) \bar{\l}, 
\end{align}
then $M$ is either flat or a round sphere. 
\end{theorem}

Theorem \ref{thm1} indicates that when 
the curvature operator of the second kind of any closed Einstein manifold satisfies the condition \eqref{1.1}, it must be nonnegative. Thus, Theorem \ref{thm1} serves as a gap theorem for Einstein manifolds under the cone condition $\mathcal{C}(2, \theta)$ imposed on the second curvature operator. Notably, since $\mathcal{C}(2, \theta)$ is strictly weaker than  $\mathcal{C}(2, 0)$ for $\theta>0$, Theorem \ref{thm1} improves the previous result from  Cao-Gursky-Tran \cite{CGT23}, Nienhaus-Petersen-Wink\cite{NPW22}, Li \cite{Li22JGA, Li24}, and Dai-Fu \cite{DF24}, that any closed Einstein manifold must be either flat or a round sphere if  $\mathring{R}$ is two-nonnegative (i.e. $\theta(n)=0$).

The proof of Theorem \ref{thm1} relies on demonstrating the nonnegativity of 
$\langle \Delta R, R \rangle$ 
for Einstein metrics under the curvature condition $\mathcal{C}(2, \theta)$. In dimensions $4$ and $5$,  Dai-Fu \cite{DF24} provide explicit expressions for $\langle \Delta R, R \rangle$ in terms of $\mathring{R}$, allowing us to directly establish its nonnegativity under the
$\mathcal{C}(2, \theta)$ condition. For dimensions $n\ge 6$, however, we must first derive the expression for  $\langle \Delta R, R \rangle$ and  establish its nonnegativity. Due to technical constraints in our method, this approach is currently restricted  to $n\ge 8$. The classification of Einstein manifolds of dimensions $6$ and $7$ under the condition $\mathcal{C}(2, \theta)$ for  positive $\theta$ remains an interesting  problem. Additionally, the sharpness of $\theta(n)$ given in  Theorem \ref{thm1} warrants further investigation.

 The paper is organized as follows:  In Section \ref{sec 2}, we give the definition of the curvature operator of the second kind and derive some identities for the curvatures on Einstein manifolds.
Section \ref{sec 3} is devoted to the proof of Theorem \ref{thm1}.

\section{Preliminaries}\label{sec 2}

\subsection{Notation}
In this section, we set up  some  notation and recall some fundamental facts about the curvature operator of the second kind, and refer to \cite{BK78, CV60, CW24, Kashiwada93,Li21, Nishikawa86} for more details.

Let $(M^n, g)$ be an $n$-dimensional Riemannian manifold. For each $p\in M$, let $V=T_p M$ and $\displaystyle \{e_i\}_{i=1}^n$ be an orthonormal basis of $V$ with respect to $g$. We always identify $V$ with its dual space $V^*$ via the metric $g$. Denote by $S^2(V)$ and $\wedge^2(V)$ the spaces of symmetric two-tensors and two-forms on $V$, respectively. The space $S^2(V)$ splits into $O(V)$-irreducible subspaces as
\begin{equation*}
        S^2(V)=S^2_0(V)\oplus \R g,
\end{equation*}
where $S^2_0(V)$ is the space of traceless symmetric two-tensors, and $g=\sum\limits_{i=1}^{n}e_i\otimes e_i$. 
Denote by $N$ the dimension of $S^2_0(V)$, that is 
\begin{equation*}
    N=\dim(S^2_0(V))=\tfrac{(n-1)(n+2)}{2}.
\end{equation*}
The space of symmetric two-tensors on $\wedge^2V$, denoted by $S^2(\wedge^2 V)$,  has an orthogonal decomposition 
    \begin{equation*}
        S^2(\wedge^2 V) =S^2_B(\wedge^2 V) \oplus \wedge^4 V,
    \end{equation*}
where $S^2_B(\wedge^2 V)$ is the space of algebraic curvature operators on $V$, consisting of all tensors $R\in S^2(\wedge^2V)$  satisfying the first Bianchi identity. 
The tensor product $\otimes$ is defined by 
\begin{equation*}
        (e_i\otimes e_j)(e_k,e_l)=\delta_{ik}\delta_{jl},
    \end{equation*}
    where 
$$\delta_{ij}=\begin{cases}
    1,  & i=j, \\ 0, &i\neq j.
\end{cases}$$
For $e_i, e_j\in V$, the symmetric product $\odot$ is defined by
 \begin{equation*}
        e_i \odot e_j=e_i\otimes e_j +e_j \otimes e_i,
    \end{equation*}
and the wedge product $\wedge$ is defined by
\begin{equation*}
        e_i \wedge e_j=e_i\otimes e_j - e_j \otimes e_i.
    \end{equation*}
Furthermore, the inner product on $\wedge^2(V)$ is defined  by 
    \begin{equation*}
        \langle A, B \rangle =\frac{1}{2}\tr(A^T B),
    \end{equation*}
and the inner product on $S^2(V)$ is defined by
    \begin{equation*}
        \langle A, B \rangle =\tr(A^T B).
    \end{equation*}
Clearly, $\{e_i \wedge e_j\}_{1\leq i<j\leq n}$ forms an orthonormal basis of $\wedge^2(V)$ and $\{\dfrac{1}{\sqrt{2}}e_i \odot e_j\}_{1\leq i<j\leq n} \cup \{\dfrac{1}{2}e_i \odot e_i\}_{1\leq i\leq n}$ forms an orthonormal basis of $S^2(V)$. 

\subsection{Curvature Operator of the Second Kind}\label{sec 2.2}
For an algebraic curvature tensor $R\in S^2_B(\wedge^2(T_p M))$, $R$ can induce two self-adjoint operators (see \cite{Nishikawa86}). 
The first one denoted by $\hat{R}:\wedge^2 (T_p M) \to \wedge^2(T_p M)$ is   defined as 
    \begin{equation*}
        \hat{R}(\omega)_{ij}=\frac{1}{2}\sum_{k,l=1}^n R_{ijkl}\omega_{kl},
    \end{equation*}
which is called  the curvature operator of the first kind. The other one,  $\overline{R}:S^2(T_p M) \to S^2(T_p M)$, is defined by
\begin{align*}
    \overline{R}(\vp)_{ij}=\sum_{k,l=1}^n R_{iklj}\vp_{kl}.
\end{align*}
The curvature operator of the second kind, introduced by Nishikawa \cite{Nishikawa86}, is a symmetric bilinear form
\begin{align*}
    \mathring{R}: S^2_0(T_pM) \times S^2_0(T_pM) \to \R
\end{align*}
obtained by restricting $\overline{R}$ to $S^2_0(T_pM)$, the space of traceless symmetric two-tensors. It was pointed out in \cite{NPW22} that the curvature operator of the second kind $\mathring{R}$ can also be interpreted as the self-adjoint operator
\begin{align}\label{2.1}
\mathring{R} = \pi \circ \overline{R}: S^2_0(T_pM) \to  S^2_0(T_pM),
\end{align}
where $\pi: S^2(T_pM) \to S^2_0(T_pM)$ is the projection map.

Let $\{\l_j\}_{j=1}^N$ be the eigenvalues of $\mathring{R}$, and let $\bar{\l}=\sum_{j=1}^N \l_j/N$ be the average of the eigenvalues of $\mathring{R}$. For any $n$-dimensional Einstein manifold, the scalar curvature, denoted by $\S$, can be expressed in terms of $\mathring{R}$ as 
\begin{align}\label{2.2}
\S=n(n-1)\bar\lambda.
\end{align}
See for instance \cite[Proposition 2.1]{CW24}.

\begin{definition}[\cite{NPW22}]
\label{def2.4}
Let $\mathcal{T}^{(0, k)}(V)$ denote the space of $(0, k)$-tensor space on $V$. For $S\in S^{2}(V)$ and $T\in\mathcal{T}^{(0, k)}(V)$, we define 
\begin{align*}
    S:\  &\mathcal{T}^{(0, k)}(V) \to \mathcal{T}^{(0, k)}(V),\\
    &(ST)({X_1}, \cdots, {X_k}) = \sum\limits_{i = 1}^k {T({X_1}, \cdots , S{X_i}, \cdots, {X_k})},
\end{align*}
and define $T^{S^{2}}\in \mathcal{T}^{(0, k)}(V)\otimes S^{2}(V)$ by
\begin{align*}
    \left\langle
    T^{S^{2}}(X_{1}, \cdots, X_{k}), S
    \right\rangle
    =(ST)\left(X_{1}, \cdots,  X_{k}\right).
\end{align*}
\end{definition}
According to the above definition, if $\{\bar S^{j}\}_{j=1}^N$
 is an orthonormal basis for $S^{2}(V)$, then 
$$T^{S^{2}}=\sum\limits_{j=1}^N\bar S^{j}T\otimes\bar S^{j}.$$
Similarly, we define $T^{S^{2}_{0}}\in \mathcal{T}^{(0, k)}(V)\otimes S^{2}_{0}(V)$ by
$$T^{S^{2}_{0}}=\sum\limits_{j=1}^N S^{j}T\otimes S^{j},$$
where $\{ S^{j}\}_{j=1}^N$ is an orthonormal basis for $S^{2}_{0}(V)$.

\subsection{A formula for Weyl tensor on Einstein manifolds}\label{sec:Weyltensor}
Given $A, B\in S^2(V)$, their Kulkarni-Nomizu product gives rise to $A \owedge B  \in  S^2_B(\wedge^2 V)$ via
\begin{equation*}
    (A \owedge B )_{ijkl} =A_{ik}B_{jl}+A_{jl}B_{ik} -A_{jk}B_{il}-A_{il}B_{jk}.
\end{equation*}
It is well known that the Riemann curvature tensor $R$ can be decomposed into irreducible components (cf. \cite[(1.79)]{CaM20}) as
\begin{align}\label{2.3}
    R=W+\frac{1}{n-2}\Ric\owedge g-\frac{\S}{2(n-1)(n-2)}g\owedge g,
\end{align}
where $W$ is the Weyl curvature tensor and $\Ric$ is the Ricci curvature tensor.   
In terms of any basis, we have
\begin{align*}
        R_{ijkl}=&W_{ijkl}+\frac{1}{n-2}\lf(R_{ik} g_{jl}+R_{jl} g_{ik}-R_{il} g_{jk}-R_{jk} g_{il}\ri)\\
        &-\frac{\S}{(n-1)(n-2)}\lf(g_{ik}g_{jl}-g_{il}g_{jk}\ri).
    \end{align*}
Moreover, if $(M^n, g)$ is an Einstein manifold, then $\Ric=\frac{\S}{n}g$, and  equation \eqref{2.3} simplifies to
\begin{align}\label{2.4}
    R=W+\frac{\S}{2n(n-1)}g\owedge g.
\end{align}

\section{Proof of Theorem \ref{thm1}}\label{sec 3}
The proof of Theorem \ref{thm1} relies on establishing the key inequality
$\langle \Delta R, R\rangle\ge 0$,
where $R$ denotes the Riemann curvature tensor.  Our approach begins by expressing $\langle \Delta R, R\rangle$ in terms of the curvature operator of the second kind  under the cone condition \eqref{1.1} for Einstein manifolds. 

Let us recall the necessary notation: $\mathring{R}$ is the curvature operator of the second kind on $(M^n, g)$, $N=(n+2)(n-1)/2$ is the dimension of $S_0^2(T_pM)$,
$\l_1\le \l_2\le \cdots\le \l_N$ are the eigenvalues of $\mathring{R}$, and $\bar\l =\sum_{i=1}^N \l_i/N$. 
We begin with some lemmas concerning  $ \langle \Delta R, R\rangle$ for 
$n\ge 8$.

\begin{lemma}\label{lm 4.1}
Let $(M^n,g)$ be an  $n$-dimensional Einstein manifold. Denote by $R$ and $W$ the Riemann curvature tensor and the Weyl tensor (see Section \ref{sec:Weyltensor}). Then the following identities hold:
\begin{align}\label{4.1}
    |R|^2=|W|^2+2n(n-1)\bar\lambda^2,
\end{align}
and
\begin{align}\label{4.2}
    \sum_{j=1}^N \lambda_j^2
    = \frac{3}{4} \lf| R \ri|^2 - (n-1)^2 \bar{\lambda}^2.
\end{align}
\end{lemma}
\begin{proof}
See  \cite[Lemma 3.1]{CW24}.
\end{proof}
By applying  the curvature condition \eqref{1.1}, we have the following estimate.
\begin{lemma}\label{lem 4.2}
Let $(M^n, g)$ be an  Einstein manifold of dimension $n\ge 6$. Suppose  $\l_1+\l_2\ge-2\theta\bar\l$ for $\theta\ge 0$, then
   \begin{align}\label{4.3}
    \sum_{j=1}^N
    \lambda_{j} |S^{j} W|^2 
    \ge -\frac{16(N-3)}{3n}\theta \bar\lambda \sum_{j=1}^N \lambda_j^2 
    +\frac{16N(N-3)}{3n} \theta \bar\lambda^3,
\end{align}
where   $S^{j} W$ is defined by Definition \ref{def2.4}.
\end{lemma}
\begin{proof}
Denote by  $$|S^{\a} W|=\max\limits_{1\le i \le N}|S^{i} W|,$$ 
and using $\l_1\le \l_2\le\cdots\le \l_N$ we estimate  
\begin{align}\label{3-4}
\begin{split}
      \sum_{j=1}^N
    \lambda_{j} |S^{j} W|^2
    &\ge \sum_{j=1}^2\l_j|S^{j} W|^2+\l_3\sum_{j=3}^N|S^{j} W|^2\\
    &=\sum_{j=1}^2(\l_j-\l_3)|S^{j} W|^2+\l_3\sum_{j=1}^N|S^{j} W|^2\\
    &\ge|S^{\a} W|^2\sum_{j=1}^2(\l_j-\l_3)+\l_3\sum_{j=1}^N|S^{j} W|^2\\
    &=|S^{\a} W|^2(\l_1+\l_2)+(\sum_{j=1}^N|S^{j} W|^2-2|S^{\a} W|^2)\l_3.
\end{split}
  \end{align}
Recall from  (4.1) and (4.5) of \cite{DF24} that 
\begin{align*}
\sum_{j=1}^{N}|S^jW|^2=\frac{2(n^2+n-8)}{n}|W|^2=\frac{4N-12}{n}|W|^2,
\end{align*}
and 
\begin{align*}
|S^jW|^2\le\frac{8n-16}{n}|W|^2, \quad 1\le j\le N,
\end{align*}
we have
\begin{align}\label{eq:SW}
    |S^{\a} W|^2(\l_1+\l_2)\ge-2\theta\bar\l|S^{\a} W|^2\ge-2\theta\bar\l\frac{8n-16}{n}|W|^2, 
\end{align}
and 
\begin{align}\label{3.5}
  \sum_{j=1}^N|S^{j} W|^2-2|S^{\a} W|^2 \ge\frac{4(N-4n+5)}{n}|W|^2= \frac {2(n^2-7n+8)}{n}|W|^2 \ge 0, 
\end{align}
where we used   $n\ge 6$ in \eqref{3.5}.
Thanks to \eqref{3.5}, $\l_3\ge (\l_1+\l_2)/2$, and the curvature condition $\l_1+\l_2\ge-2\theta\bar{\l}$, we deduce from \eqref{3-4} that  
\begin{align}\label{4.4}
(\sum_{j=1}^N|S^{j} W|^2-2|S^{\a} W|^2)\l_3\ge-\theta\bar\l\frac{2(n^2-7n+8)}{n}|W|^2.
\end{align}
From \eqref{4.1} and \eqref{4.2}, we have 
\begin{align}  \label{3-7}  
|W|^2    
= \frac{4}{3} \sum_{j=1}^N \lambda_j^2    
-\frac{2}{3}(n-1)(n+2)\bar\lambda^2=\frac{4}{3} \sum_{j=1}^N \lambda_j^2    
-\frac{4N}{3}\bar\lambda^2.
\end{align}
Substituting \eqref{3-7} into \eqref{eq:SW} and \eqref{4.4} yields inequality \eqref{4.3}.
\end{proof}

\begin{lemma}\label{lm 3.3}
Let $(M^n, g)$ be an  Einstein manifold of dimension $n\ge 6$. Suppose  $\l_1+\l_2\ge-2\theta\bar\l$ for $\theta\ge 0$, then
\begin{align}\label{4.5}
\begin{split}
    \frac{9n}{16} \langle \Delta R, R \rangle
    \ge & \Big[ N(N-3)\theta-(2N-9n+6)N \Big]\bar\l^3\\
   &+  \Big[ (2N-12n+6)-(N-3)\theta \Big] \bar\lambda \sum_{j=1} ^{N}\lambda_{j}^2+3n \sum_{j=1}^{N}\lambda_{j}^3,
    \end{split}
\end{align}
where $\Delta$ is the Laplace-Beltrami operator with respect to the metric $g$. 
\end{lemma}

\begin{proof}
Let
$\S$  be the scalar curvature of $M$. Recall from equality (3.3) of \cite{DF24}  that
\begin{align*}
    3\langle \Delta R,R\rangle
  =&\displaystyle\sum_{j=1}^N {{\lambda _j }{{\left| {{S^j }W} \right|}^2}}+8\left(\frac{-n^3+6n^2+12n-8}{3n^4(n-1)^2} \right){\S^3}\\
   &+8\left(\frac{2n^2-22n+8}{3n^2(n-1)} \right)\S\displaystyle \sum_{j=1}^N  {\lambda _j ^2}  + 16\displaystyle \sum_{j=1}^N {\lambda_j ^3}.
\end{align*}
 Substituting  $\S=n(n-1)\bar \lambda$ and $N=(n-1)(n+2)/2$ into the above equality yields
 \begin{align*}
    3\langle \Delta R,R\rangle
  =&\sum_{j=1}^N {{\lambda _j }{{\left| {{S^j }W} \right|}^2}}
  -\frac{16N(2N-9n+6)}{3n}\bar\lambda^3\\
   &+\frac{16(2N-12n+6)}{3n} \bar\lambda  \sum_{j=1}^N \lambda _j ^2  + 16\sum_{j=1}^N {\lambda _j ^3}.
\end{align*}
Using  \eqref{4.3}, we estimate  that
\begin{align*}
    3\langle \Delta R,R\rangle
    \ge & 
    -\frac{16(N-3)\theta}{3n} \bar\lambda \sum_{j=1}^N  \lambda_j^2 +\frac{16N(N-3)}{3n} \theta \bar\lambda^3+ 16\sum_{j=1}^N {\lambda _j ^3}\\
     &-\frac{16N(2N-9n+6)}{3n}\bar\lambda^3
   +\frac{16(2N-12n+6)}{3n} \bar\lambda  \sum_{j=1}^N \lambda _j ^2   \\
    =\quad& \frac{16}{3n}\Big[ N(N-3)\theta-(2N-9n+6)N \Big]\bar\l^3\\
   &+ \frac{16}{3n} \Big[ (2N-12n+6)-(N-3)\theta \Big] \bar\lambda \sum_{j=1} ^{N}\lambda_{j}^2+16 \sum_{j=1}^{N}\lambda_{j}^3,
\end{align*}
proving \eqref{4.5}. 
\end{proof}

The following proposition is a crucial step in the proof of the main theorem, establishing the nonnegativity of $\langle \Delta R, R\rangle$ under the given curvature constraints.
\begin{proposition}\label{pro 4.4}
Let $(M^n,g)$ be an Einstein manifold of dimension $n\ge 8$ or $n=4$ or $n=5$. Suppose the curvature operator of the second kind satisfies the cone condition \eqref{1.1}. Then 
\begin{align}\label{k-e}
   \left\langle \Delta R, R\right\rangle\ge 0,
\end{align}
with the equality holding if and only if 
 \begin{align*}
    \l=\left(1, 1,\cdots,1\right)\bar\l \quad  \text{or} \quad \l=\left(-\frac{2(N-1)\theta+N}{N-2}, \frac{N+2\theta}{N-2}, \cdots, \frac{N+2\theta}{N-2}\right)\bar\l.
\end{align*}
Here, $\bar \l=\sum_{j=1}^N \l_j/N$ is the average of $\l_j$.
\end{proposition}
\begin{proof}
We prove the case of $n\ge 8$. Let $\l=(\l_1, \l_2,\cdots,\l_N)$ and  denote the right hand side of \eqref{4.5} by $f(\l)$, namely
\begin{align*}
\begin{split}
    f(\l):=&\frac{16}{3n}\Big[ N(N-3)\theta-(2N-9n+6)N \Big]\bar\l^3\\
   &+  \frac{16}{3n}\Big[ (2N-12n+6)-(N-3)\theta \Big] \bar\lambda \sum_{j=1} ^{N}\lambda_{j}^2+16 \sum_{j=1}^{N}\lambda_{j}^3,
\end{split}
\end{align*}
where 
$$\theta(n)=\frac{2N-9n+6}{N+6n-3}$$
 by \eqref{1.2}.
Using the identities
$$
(N-3)\theta-(2N-9n+6)=-6n\theta,
$$
and
$$
(2N-12n+6)-(N-3)\theta=6n\theta-3n,
$$
we simplify   $f(\l)$ to
\begin{align}
    f(\l)=16\left[-2N\theta \bar\l^3+(2\theta-1)\bar\lambda \sum_{j=1} ^{N}\lambda_{j}^2+ \sum_{j=1}^{N}\lambda_{j}^3\right].
\end{align}
From the curvature condition \eqref{1.1}, we have $\bar \l\ge 0$. If $\bar \l=0$, then $\l_1=\l_2=\cdots=\l_N=0$, and \eqref{k-e} holds trivially due to \eqref{4.5}. For the case $\bar \l>0$, perform a change of variables:
 \begin{align*}
   x_j=\l_j/\bar\l,\quad  x=(x_1, x_2, \cdots, x_N), \quad \text{and}\quad  \beta=1+\theta. 
 \end{align*}
This allows us to express $f(\l)$ as
\begin{align*}
f(\l)=16 \bar\l^3 F(x),
\end{align*}
 where
\begin{align*}
F(x)=\sum_{i=1}^{N}x_i^3-(3-2\beta)\sum_{i=1}^Nx_i^2+2N(1-\beta).
\end{align*} 
Thus it suffices to show:
$$F(x)\ge 0$$
under the constraints
    \begin{align}\label{4.6}   
 \sum_{j=1}^N x_j=N, \quad \text{and}\quad  x_i+x_j\ge -2(\beta-1)\quad \text{for}\quad  i\neq j, 
    \end{align}
and the minimizers being
$$
(1,1,\cdots, 1) \quad \text{and} \quad \left(-\frac{2(N-1)(\beta-1)+N}{N-2}, \frac{N+2(\beta-1)}{N-2}, \cdots, \frac{N+2(\beta-1)}{N-2}\right).
$$

We now use the method of Lagrange multipliers to prove $F(x)\ge 0$  under the given  constraints \eqref{4.6}.

{\bf Step 1.} Identify the possible critical points within the domain. Consider the set
\begin{align*}
  \Omega=\{x\in \R^N: \sum_{j=1}^N x_j=N, \text{\quad $x_i+x_j>-2(\beta-1)$ for all $i\neq j$}\},  
\end{align*}
and define the Lagrangian function as
$$
\mathcal{L}(x, \xi):=F(x)+\xi(\sum_{i=1}^Nx_i-N).
$$
By computing the derivatives of  $\mathcal{L}$,  the possible critical points must satisfy 
\begin{align*}
\begin{cases}
    3x_i^2-2(3-2\b)x_i+\xi=0,\\
    \displaystyle\sum_{i=1}^Nx_i=N.
\end{cases}
\end{align*}
Hence, the possible critical points are of the form: 
\begin{align*}
Q_{m}=(\underbrace{r, \cdots, r}_{m}, \underbrace{s, \cdots, s}_{N-m}), \quad \text{for $m\le N/2$,}
\end{align*}
 where $r$ and $s$ are determined by the equations:
\begin{align}\label{4.11}
    \begin{cases}
        r+s=\frac{2(3-2\b)}{3}, \\
        mr+(N-m)s=N.
    \end{cases}
\end{align}
If $m=N/2$,  equations \eqref{4.11} are unsolvable for  $\b>1$. 
For $m<\frac{N}{2}$, solving the system  \eqref{4.11} yields 
\begin{align*}
r=A-\frac{N-NA}{N-2m},\quad s=A+\frac{N-NA}{N-2m},
\end{align*}
where $A=(3-2\b)/3<1$ (since $\b>1$).
Substituting these into $F$, we obtain
\begin{align*}
    F(Q_m)=-2NA^3-3A^2(N-NA)+\frac{(N-NA)^3}{(N-2m)^2}+2N(1-\b),
\end{align*}
which implies
\begin{align}\label{4.12}
    F(Q_m)> F(Q_0)=F(1,\cdots,1)=0, \quad 1\le m\le N/2.
\end{align}

{\bf Step 2.} Now we deal with the boundary case:  $x_1+x_2=-2(\b-1)$, $\sum_{j=1}^N x_j=N$, and $x_1\le x_2\le \cdots\le x_N$. 

To do this, we fix $x_2$ first: let $x_2=a$, and then $x_1=-2(\b-1)-a$. So by the constraints $x_1\le x_2\le \cdots \le x_N$ and $\sum_{j=1}^N x_j=N$, we have
\begin{align}\label{3-14}
     -(\b-1)\le a\le \frac{N-2+2\b}{N-2}.
\end{align}
For $1\le k \le N-2$, let $\Omega_k(a)$ be the set:
\begin{align*}
    \begin{cases}
        x_1=-2(\b-1)-a, \\
        x_2=x_3=\cdots=x_{N-k}=a,\\
        x_j>a, \quad j>N-k,\\
        \sum_{j=3}^N x_j=N+2(\b-1),
    \end{cases}
\end{align*}
 and the Lagrangian function is 
$$
\mathcal{L}(x_{N-k+1},\cdots, x_N, \mu)=F(x)+\mu\big(\sum_{j=3}^Nx_j-N+2-2\b\big),
$$
and by calculating the derivatives,  the possible critical points
$$x=\big(-2\b+2-a, \underbrace{a, \cdots, a}_{N-1-k}, x_{N-k+1}, \cdots, x_N 
    \big)$$
satisfy
    \begin{align*}
   \begin{cases}
       3 x_j^2-2(3-2\b) x_j + \mu=0,\quad j>N-k,\\
       \sum_{j=3}^N x_j=N+2\b-2.
   \end{cases} 
\end{align*}
So the possible critical points are given  by
\begin{align*}
  P_{k, l}(a):=  \big( -2\b+2-a,  \underbrace{a,  \cdots, a}_{N-1-k}, \underbrace{c, \cdots, c}_{l}, \underbrace{d, \cdots, d}_{k-l} \big) \quad \text{for $0\le l\le \frac{k}{2}$,}
\end{align*}
 where $c$ and $d$ are given by 
\begin{align}\label{4.7}
    \begin{cases}
        c+d=\frac{2(3-2\b)}{3},\\
        lc+(k-l)d=N-2+2\b-(N-2-k)a.
    \end{cases}
\end{align}
To simplify calculations, we denote 
\begin{align*}
    A=\frac{3-2\b}{3} \quad \text{and}\quad B=N-2+2\b-(N-2-k)a.
\end{align*}

When $l=k/2$ and equations \eqref{4.7} are solvable, it shall hold $B-kA=0$. But using  $a\le (N-2+2\b)/(N-2)$ (see \eqref{3-14}), we estimate 
\begin{align}\label{3-17}
   B-kA\ge N-2+2\b-(N-2-k)\frac{N-2+2\b}{N-2}-kA=\frac{2k\b(N+1)}{3(N-2)}>0.  
\end{align}
Thus we conclude that \eqref{4.7} are unsolvable when $l=k/2$.

 When $l<k/2$, by solving \eqref{4.7} we derive
$$c=A-\frac{B-kA}{k-2l},\quad \text{and}\quad d=A+\frac{B-kA}{k-2l}.$$
Then we calculate $F$ at $P_{k, l}(a)$ that
\begin{align*}
    F(P_{k, l}(a))
    =&(-2\b+2-a)^3+(N-1-k)a^3\\
    &-3A\left[(-2\b+2-a)^2+(N-1-k)a^2\right]\\
    &+2N(1-\b)-2kA^3-3A^2(B-kA)+\frac{(B-kA)^3}{(k-2l)^2},
\end{align*}
hence by  \eqref{3-17}, we get
\begin{align}\label{4.9}
    F(P_{k, l}(a))>F(P_{k, 0}(a)), \quad \text{for $ 1\le l<k/2$,}
\end{align}
here
$$
P_{k, 0}(a)=\Big(-2\b+2-a, \underbrace{a, \cdots, a}_{N-1-k}, \underbrace{B/k, \cdots, B/k}_k\Big).
$$

Substituting $P_{k,0}(a)$ into $F$, we get
\begin{align}\label{3-19}
 F(P_{k, 0}(a))=2 (a-1) (a+\b-1)D+D^2 \left(\frac{D}{k^2}+\frac{3a-3+2\b}{k}\right).
\end{align}
Here, $D=N-2+2\b-(N-2)a$. By \eqref{3-14}, it follows that $D\ge0$.
Let $g(s)=Ds^2+(3a-3+2\b)s$, and 
$$
g'(\frac 1 k)=\frac{2D}{k}+(3a-3+2\b)\ge\frac{(N+2)\b}{N-2}>0,
$$
where we have used $k\le N-2$  and $a\ge-(\b-1)$. Therefore we conclude from \eqref{3-19} that
\begin{align}\label{3-20}
   F(P_{k, 0}(a))> F(P_{N-2, 0}(a)), \quad \text{for $1\le k\le N-3$.}
\end{align}

Combining inequalities \eqref{4.9} and \eqref{3-20}, we conclude that the function $F(x)$ attains its minimum value  at 
$$
P_{N-2, 0}(a)=\left(-2\b+2-a, a, \frac{N-2+2\b}{N-2}, \cdots, \frac{N-2+2\b}{N-2}\right)
$$
under the constraints: $x_1\le x_2\le \cdots \le x_N$, $x_1+x_2=-2(\b-1)$, $x_2=a$ and  $\sum_{j=1}^N x_j=N$.

Substituting $P_{N-2,0}(a)$ into $F$ yields
\begin{align*}
    F(P_{N-2, 0}(a))=-2\b a^2+4\b (1-\b )a+C(n),
\end{align*}
with 
\begin{align*}
C(n)=2N(\b-1)+2\b(2\b-1)+\frac{8\b^3(N-1)}{(N-2)^2}.
\end{align*}
Noting that $a\in [-\b+1, \frac{N-2+2\b}{N-2}]$, we have
\begin{align}\label{4.10}
    F(P_{N-2,0}(a))> F\left(P_{N-2, 0}\left(\frac{N-2+2\b}{N-2}\right)\right)=0
\end{align}
for $-\b+1\le a< \frac{N-2+2\b }{N-2}$. Here
\begin{align*}
P_{N-2, 0}\left(\frac{N-2+2\b}{N-2}\right)
=\left(-\frac{2(\b-1)(N-1)+N}{N-2}, \frac{N-2+2\b }{N-2}, \cdots, \frac{N-2+2\b}{N-2}\right).
\end{align*}
Hence, we obtain that on the boundary of $\Omega$: $x_1+x_2=-2(\b-1)$, $F(x)$ attains its minimum value of $0$  with the minimizer at $P_{N-2, 0}\left(\frac{N-2+2\b}{N-2}\right)$.

From the above two steps,  we prove that under the constraints 
$\sum_{j=1}^N x_j=N$, $x_1\le x_2\le \cdots\le x_N$, and $x_1+x_2\ge -2(\b-1)$, the function $F(x)$ attains its minimum value of $0$, achieved at the points 
\begin{align*}(1, \cdots, 1) \quad \text{and} \quad \left(-\frac{2(N-1)\theta+N}{N-2}, \frac{N+2\theta}{N-2}, \cdots, \frac{N+2\theta}{N-2}\right).
\end{align*}
This completes the proof of Proposition \ref{pro 4.4} for $n\ge 8$.

For dimensions $n=4$ and $5$,  $\langle \Delta R, R\rangle$ 
can be expressed explicitly in terms of the curvature operator of the second kind as:
\begin{align}\label{3.22}
     \left\langle \Delta R, R\right\rangle=\begin{cases}
      8\left(\sum_{j=1}^9 \l_j^3-9\bar\l^3\right), &\quad n=4,\\ 
      8\left(\sum_{j=1}^{14} \l_j^3+\frac{1}{3}\bar\l\sum_{j=1}^{14}\l_j^2-\frac{56}{3}\bar\l^3\right), &\quad n=5.
     \end{cases}
\end{align}
See Section 4.2 of \cite{DF24}. Then an argument similar to the one used for $n\ge 8$ shows that Proposition   \ref{pro 4.4}  holds for $n=4$ and $n=5$.
 \end{proof}

Now we prove the main theorem. 
\begin{proof}[Proof of  Theorem \ref{thm1}]
By \eqref{k-e}, we have
\begin{align*}
     \Delta |R|^2
    =2|\nabla R|^2 + 2\langle \Delta R,R \rangle \ge 2|\nabla R|^2.
\end{align*}
Integrating this inequality over $M$ yields
\begin{align*}
    0=\int_M  \Delta |R|^2\, d\mu_g \ge\int_M 2|\nabla R|^2 \, d\mu_g \ge 0,
\end{align*}
which implies $|\nabla R|=0$. Therefore, $M$ is a symmetric space. 
Since $\theta(n)$, defined by \eqref{1.2} is less than $2(n-1)/(n+2)$ for $n\ge 5$ and equal to $1/2$ for $n=4$,  it follows from Theorems 1.5 and 1.8 of \cite{Li24}   that $M$ is either flat (when  $\bar \l=0$) or a rational homology sphere  (when $\bar \l>0$).

Compact symmetric spaces that are rational homology spheres were classified completely by Wolf  \cite[Theorem 1]{Wolf}. Apart from spheres, the   only simply connected example is $SU(3)/SO(3)$. 
Proposition \ref{pro 4.4} implies that the eigenvalues of $\mathring{R}$ are either $\lambda=(1,\cdots,1)\bar{\l}$ or $\left(-\frac{2(N-1)\theta+N}{N-2}, \frac{N+2\theta}{N-2}, \cdots, \frac{N+2\theta}{N-2}\right)\bar\l$. Combining this with \cite[Example 4.5]{NPW22}, we know that $M$ is a round sphere if $\bar \l>0$.
\end{proof}

\section*{Acknowledgments}
The first author expresses her gratitude to her advisor, Professor Ying Zhang, for lots of encouragement and helpful suggestions. The authors would like to thank Professor Xiaolong Li for many helpful discussions. The research of this paper is partially supported by  NSF of Jiangsu Province Grant No.BK20231309.

\bibliographystyle{plain}
\bibliography{ref}
\end{document}